\documentclass[11pt]{amsart}
\usepackage[T1]{fontenc}
\usepackage{lmodern}
\usepackage{amsmath,amssymb,amsthm,mathtools}
\usepackage{microtype}
\usepackage[colorlinks=true,linkcolor=blue,citecolor=blue,urlcolor=blue]{hyperref}
\usepackage[margin=1.1in]{geometry}

\newtheorem{theorem}{Theorem}[section]
\newtheorem{proposition}[theorem]{Proposition}
\newtheorem{lemma}[theorem]{Lemma}
\newtheorem{corollary}[theorem]{Corollary}
\newtheorem{remark}[theorem]{Remark}

\newcommand{\HH}{\mathbb H}
\newcommand{\CC}{\mathbb C}
\newcommand{\RR}{\mathbb R}
\newcommand{\EE}{\mathbf E}
\newcommand{\PP}{\mathbf P}
\newcommand{\F}{\mathcal F}

\newcommand{\Cont}{\operatorname{Cont}_{d}}
\newcommand{\Occ}{\mu_{\mathrm{O}}}
\newcommand{\Haus}{\mu_{\mathrm{H}}}
\newcommand{\dist}{\operatorname{dist}}
\newcommand{\BGE}{\begin{equation}}
\newcommand{\BGEN}{\begin{equation*}}
\newcommand{\EDE}{\end{equation}}
\newcommand{\EDEN}{\end{equation*}}

\def\til{\widetilde}

\def\sem{\setminus}
\def\lin{\overline}

\title[Hausdorff Measure as Natural Length]{An Exact Hausdorff Measure as Natural Length for SLE}
\author{Dapeng Zhan}
\date{\today}
\address{Department of Mathematics, Michigan State University,
East Lansing, MI 48824, USA}
\begin{document}

\begin{abstract}
Fix $0<\kappa<8$ and put $d=1+\kappa/8$.  Recent work by OpenAI identifies an explicit exact Hausdorff gauge
\(
 h_\kappa(r)=r^d\bigl(\log\log(1/r)\bigr)^{(2-d)/2}
\) for sufficiently small $r$
for chordal $\mathrm{SLE}_\kappa$: the corresponding Hausdorff measure is almost surely positive and finite on every segment $\gamma([s,t])$ with $0<s<t<\infty$, and has finite expectation in bounded subsets of the plane.  We show that this Hausdorff measure is a deterministic positive constant multiple of the occupation measure (natural length) constructed by Lawler and Rezaei. %, and hence that it induces natural parametrization up to a deterministic change of scale.
\end{abstract}

\maketitle

\section{Introduction}
Let $0<\kappa<8$.  Let $\gamma$ be chordal $\mathrm{SLE}_\kappa$ in the upper half-plane  $\HH$ from $0$ to $\infty$ \cite{RS}, parametrized by half-plane capacity $2t$, and let $d=1+\frac{\kappa}{8}$ be its Hausdorff dimension \cite{Beffara}.  Write $\Gamma_t=\gamma([0,t])$ and $\Gamma=\gamma([0,\infty))$.

The recent exact-gauge theorem \cite{OAIexact} states that the gauge $h_\kappa$, specified for sufficiently small $r>0$ by
\BGE
 h_\kappa(r)=r^d\bigl(\log\log(1/r)\bigr)^{(2-d)/2}, \label{gauge}
\EDE
gives a positive finite Hausdorff measure on every nontrivial positive-time compact segment of the SLE trace.  It also gives finite expected $\mathcal H^{h_\kappa}$-measure of the whole trace in every bounded disk.  The purpose of this note is to identify the resulting measure.

For a gauge function $h$ and a Borel set $E\subset\CC$, let
$\mathcal H^h_E$ denote the Hausdorff measure with gauge $h$
on $E$, extended by zero to $\CC$.
These measures are consistent under restriction:
if $E\subset F$ are Borel sets, then
$\mathcal H^h_E=\mathcal H^h_F(\cdot\cap E)$.

Lawler and Rezaei \cite{LRMinkowski} proved that the $d$-dimensional Minkowski content of $\mathrm{SLE}_\kappa$ exists for every $0<\kappa<8$ and is a deterministic normalization of the natural parametrization constructed in \cite{LSnatural,LZnatural}.  Their normalization will be used below.  Let $\Occ$ denote their occupation measure on the SLE curve $\Gamma$, constructed in \cite[Theorem 3.1 and Proposition 3.3]{LRMinkowski} and extended by zero to $\CC$.  It is the pushforward of Lebesgue measure on $\RR_+$ under the naturally parametrized curve.  Its relation to Minkowski content on curve segments and spatial restrictions is recorded in Section~2.

Our main result is the following.

\begin{theorem} \label{thm:main}
For each fixed $0<\kappa<8$ there is a deterministic constant $C_\kappa\in(0,\infty)$ such that, almost surely,
 $\mathcal H^{h_\kappa}_\Gamma=C_\kappa\,\Occ$. In particular, almost surely, simultaneously for all $0\le s<t<\infty$, $\mathcal H^{h_\kappa}_\Gamma(\gamma([s,t]))  = C_\kappa\,\Cont(\gamma([s,t]))$.
\end{theorem}

\section{Inputs and notation}
We use the following consequence of \cite[Theorem 1.1]{OAIexact}.

\begin{theorem}\label{thm:exact-input}
For fixed $0<\kappa<8$, almost surely,
\begin{equation}\label{eq:positive-finite-segments}
 0<\mathcal H^{h_\kappa}_\Gamma(\gamma([s,t]))<\infty,
 \qquad 0<s<t<\infty.
\end{equation}
Moreover,  
\begin{equation}\label{eq:local-expectation}
 \EE\bigl[\mathcal H^{h_\kappa}_\Gamma(B(0,R))\bigr]<\infty,\quad\text{for every }R\in(0,\infty).
\end{equation}
\end{theorem}

Let $G$ denote the chordal SLE Green function with the normalization of Lawler and Rezaei \cite{LRMinkowski}:
\begin{equation}\label{eq:green}
 G(z)=c_\kappa^{\rm gr}\,[\Im z]^{d-2}[\sin\arg z]^{8/\kappa-1},
\end{equation}
where the positive constant $c_\kappa^{\rm gr}$ is fixed by their one-point Minkowski-content limit.   If $D$ is simply connected with marked prime ends $a,b$ and $F:D\to\HH$ sends $a$ to $0$ and $b$ to $\infty$, then
\BGE
 G_D(z;a,b)=|F'(z)|^{2-d}G(F(z)). \label{GD}
\EDE

Let $H_t$ be the unbounded component of $\HH\setminus\Gamma_t$, let $g_t:H_t\to\HH$ be hydrodynamically normalized, and let $U_t=\sqrt\kappa B_t$ be its chordal driving function.  Throughout, $(\F_t)$ is the usual augmentation of the natural filtration of $B$.  In $G_{H_t}(z;\gamma(t),\infty)$, the marking $\gamma(t)$ denotes the current tip prime end mapped by $g_t$ to $U_t$.  Let $A$ denote the planar Lebesgue measure. Put
\BGE
 \Psi_t(D)=\int_{D\cap H_t}G_{H_t}(z;\gamma(t),\infty)\,A(dz). \label{Psi}
\EDE
Put $\Theta_t^{\rm O}=\Occ(\Gamma_t)$.  By \cite[Theorem 1.1 and Section 3.2, equation (30)]{LRMinkowski}, $\Theta^{\rm O}_t$ is continuous and strictly increasing, and almost surely, simultaneously for all $0\le s<t<\infty$,
\begin{equation}\label{eq:LR-segment}
 \Theta_t^{\rm O}-\Theta_s^{\rm O}
 =\Occ(\Gamma_t\setminus\Gamma_s)
 =\Occ(\gamma([s,t]))=\Cont(\gamma([s,t])).
\end{equation}
This includes boundary visits; see also \cite[Lemma 3.7]{LRMinkowski}.  In capacity time the occupation measure is $\gamma_*(d\Theta^{\rm O})$, where $d\Theta^{\rm O}$ is the Lebesgue--Stieltjes measure of the natural clock.  Equation \eqref{eq:LR-segment} gives zero occupation mass to overlaps of separated time intervals.  Consequently, for every Borel set $D$,
$\Theta_t^{\rm O}(D):=\Occ(D\cap\Gamma_t)=\int_0^t\mathbf1_D(\gamma(s))\,d\Theta_s^{\rm O}$.
The integral representation shows that $\Theta^{\rm O}(D)$ is adapted; it is continuous because its increments are
bounded by those of $\Theta^{\rm O}$.
By uniqueness of measures, the first-moment identity in
\cite[Theorem 3.1]{LRMinkowski} extends from dyadic squares
to all Borel sets, giving
$\EE[\Occ(D)]=\int_{D\cap\HH}G(z)\,A(dz)$.

We next identify the occupation mass of spatial restrictions with their Minkowski content.  Fix a bounded open set $D\Subset\HH$ with piecewise smooth boundary.  The first-moment identity gives $\Occ(\partial D)=0$ almost surely.  On this event, fix $T>0$ and partition $[0,T]$ into finitely many intervals $I_j=[t_{j-1},t_j]$.  For $E=D$ or $D^c$, the segments $\gamma(I_j)$ meeting $E$ cover $\Gamma_T\cap E$.  Finite subadditivity of upper Minkowski content and \eqref{eq:LR-segment} therefore give
\[
 \overline{\operatorname{Cont}}_d(\Gamma_T\cap E)
 \le\sum_{j:\,\gamma(I_j)\cap E\ne\varnothing}
       (\Theta_{t_j}^{\rm O}-\Theta_{t_{j-1}}^{\rm O}).
\]
Along any sequence of partitions whose mesh tends to zero, continuity of $\gamma$ and dominated convergence with respect to $d\Theta^{\rm O}$ show that the sum tends to $\Occ(\Gamma_T\cap E)$; the only possible discrepancy occurs at times with $\gamma(s)\in\partial D$, which have zero $d\Theta^{\rm O}$-measure.
Since $\Cont(\Gamma_T)=\Occ(\Gamma_T)$, the decomposition of $\Gamma_T$ into its parts in $D$ and $D^c$ also gives
$\underline{\operatorname{Cont}}_d(\Gamma_T\cap D)\ge
\Cont(\Gamma_T)-\overline{\operatorname{Cont}}_d(\Gamma_T\cap D^c)
\ge\Occ(\Gamma_T\cap D)$.
Thus $\Cont(\Gamma_T\cap D)=\Occ(\Gamma_T\cap D)$, simultaneously for all $T$.
Transience implies that all visits to the bounded set $D$ occur before a finite time, so also $\Cont(\Gamma\cap D)=\Occ(D)$.
The conditional-expectation identity in \cite[Theorem 1.1]{LRMinkowski} now yields
\begin{equation}\label{eq:LR-doob-meyer}
 \Theta_t^{\rm O}(D)+\Psi_t(D) \quad\text{is an }\mathcal F\text{-martingale.}
\end{equation}

\section{Regular variation and conformal covariance}

We use the conformal change-of-variables formula of \cite[Proposition 2]{Chistyakov2003}.

\begin{lemma} \label{lem:conformal-covariance}
Let $h$ be a gauge such that for some $q>0$,
\begin{equation}\label{eq:regular-variation}
 \lim_{r\downarrow0}\frac{h(ar)}{h(r)}=a^q
 \qquad\text{for every }a>0.
\end{equation}
Let $F:U\to V$ be conformal.  Then for every Borel set $E\subset U$ and nonnegative Borel function $f$ on $V$,
\begin{equation}\label{eq:area-formula}
 \int_{F(E)} f(w) \mathcal H^h_{F(E)} (dw)
 =\int_E f(F(z)) |F'(z)|^q\, \mathcal H^h_E(dz).
\end{equation}
% where $\mathcal H^h_E$ and $ \mathcal H^h_{F(E)}$ are  the Hausdorff measures with gauge $h$ on $E$ and $F(E)$, respectively.
\end{lemma}

Although stated for integrable functions,
the proof of \cite[Proposition 2]{Chistyakov2003} also applies verbatim to nonnegative measurable functions,
allowing infinite values.

The gauge $h_\kappa$ in \eqref{gauge} satisfies Lemma \ref{lem:conformal-covariance} with $q=d$.  In particular, for every Borel set $E$,
\begin{equation}\label{eq:exact-scaling-H}
 \mathcal H^{h_\kappa}_{a\Gamma}(aE)=a^d\mathcal H^{h_\kappa}_\Gamma(E),\qquad a>0.
\end{equation}
\begin{lemma}  $\mathcal H^{h_\kappa}_\Gamma(\RR)=0$. \label{lem:real-zero}
\end{lemma}
\begin{proof}
Take $d'\in (1,d)$. Since $h_\kappa(r)/r^{d'}\to0$ as $r\downarrow0$, we get
$$\mathcal H^{h_\kappa}_\Gamma(\RR)=\mathcal H^{h_\kappa}_{\Gamma}(\Gamma\cap\RR)=\mathcal H^{h_\kappa}_{\RR}(\Gamma\cap\RR)\le \mathcal H^{d'}_{\RR}(\RR)=0.$$
\end{proof}

\section{The first-visit Hausdorff clock}

Let $\Haus=\mathcal H^{h_\kappa} _\Gamma$.  It is almost surely locally finite by \eqref{eq:local-expectation}.  For Borel $D\subset\CC$, put $\Theta_t^{\rm H}(D)=\Haus(D\cap\Gamma_t)$ and write $\Theta_t^{\rm H}=\Haus(\Gamma_t)$.  By monotone convergence, $\Theta_\infty^{\rm H}(D):=\lim_{t\to\infty}\Theta_t^{\rm H}(D)=\Haus(D)$; this terminal variable is integrable for bounded $D$ by \eqref{eq:local-expectation}.

\begin{lemma} \label{lem:clock-continuity}
Let $D\Subset\HH$ be open.  Then the process $t\mapsto\Theta_t^{\rm H}(D)$ is adapted, nondecreasing, and continuous.
\end{lemma}

\begin{proof}
For a deterministic compact $K$, the quantity $\mathcal H^{h_\kappa}_\Gamma(K\cap\Gamma_t)$ is $\F_t$-measurable: Hausdorff measure of a compact set can be computed by taking infima over a countable family of finite open covers whose members are finite unions of rational balls. Exhausting $D$ from inside by deterministic compact sets gives the same conclusion for $D$.  Thus $\Theta_t^{\rm H}(D)$ is adapted.

For continuity, define $\sigma(z)=\inf\{t\ge0:z\in\Gamma_t\}$ on $\Gamma$.  This map is Borel since $\{\sigma\le t\}=\Gamma_t$ is closed.  Since $\Haus$ gives zero mass to singletons, the finite pushforward measure $\sigma_*(\Haus|_D)$ is atomless because $\sigma^{-1}(\{t\})\subset\{\gamma(t)\}$.  Its distribution function $\Theta_t^{\rm H}(D)=\Haus(D\cap\{\sigma\le t\})$ is therefore continuous.
\end{proof}

\begin{lemma} \label{lem:future-measure}
Fix $t\ge0$ and set  $F_t(w)=g_t^{-1}(w+U_t)$, which is a conformal map from $\HH$ onto $H_t$ and a.s.\ extends continuously to $\overline\HH$.
Let $\widetilde\gamma$ be the chordal Loewner curve driven by $W_s:=U_{t+s}-U_t$.  By independent Brownian increments, $\widetilde\gamma$
is a chordal $\mathrm{SLE}_\kappa$ in $\HH$ from $0$ to $\infty$,
independent of $\F_t$. Let $\widetilde\mu_{\rm H}=\mathcal H^{h_\kappa}_{\widetilde \Gamma}$, where $\widetilde\Gamma:=\widetilde\gamma([0,\infty))$.  Then, for every nonnegative Borel function $f$ on $\CC$,
\begin{equation}\label{eq:future-pushforward}
 \int_{\Gamma\setminus\Gamma_t}f(z)\,\Haus(dz)
 =\int_{\widetilde\Gamma}f(F_t(w))|F_t'(w)|^d\,\widetilde\mu_{\rm H}(dw)
 \quad\text{a.s.}
\end{equation}
\end{lemma}

\begin{proof}
 SLE theory yields $ \gamma([t,\infty))=F_t(\widetilde\Gamma)$. Thus, $(\Gamma\setminus \Gamma_t)\cap H_t=F_t (\widetilde \Gamma\cap \HH)$. Lemma \ref{lem:conformal-covariance} gives \eqref{eq:future-pushforward} with the left integral restricted to $H_t$ and right integral restricted to $\HH$. We also get  $\til\Gamma\subset\lin\HH$, $\Gamma\setminus\Gamma_t\subset \gamma([t,\infty))=F_t(\widetilde\Gamma)\subset \overline{H_t}$ and $\partial H_t\subset \Gamma_t\cup \RR$. Thus, $(\Gamma\setminus\Gamma_t)\setminus H_t$ and $\til\Gamma\sem \HH$ are subsets of $\RR$ and are null by Lemma \ref{lem:real-zero}. Removing the restriction to $H_t$ and $\HH$ therefore leaves the integrals unchanged.
\end{proof}

\section{The mean Hausdorff measure is Green measure}

Define the mean measure  by $\nu(E)=\EE[\Haus(E)]$ for Borel $E$. Since $\Gamma\subset\overline\HH$ and
$\nu(\RR)=0$ by Lemma \ref{lem:real-zero}, we henceforth regard $\nu$ as a measure on $\HH$.  Local finiteness follows from \eqref{eq:local-expectation}.  Brownian scaling of SLE together with \eqref{eq:exact-scaling-H} gives
\begin{equation}\label{eq:nu-scaling}
 \nu(rE)=r^d\nu(E),\qquad r>0.
\end{equation}
In particular, $\nu(B(0,R)\cap\HH)=O(R^d)$ as $R\to\infty$.
Reflection of chordal SLE in the imaginary axis also gives
\begin{equation}\label{eq:nu-reflection}
 \nu(-\overline E)=\nu(E).
\end{equation}

\begin{lemma} \label{lem:hypo-smoothing}
The measure $\nu$ has a density $\rho\in C^\infty(\HH)$ with respect to $A$, and $\rho$ satisfies
\begin{equation}\label{eq:rho-pde}
 \frac\kappa2\rho_{xx}
 +\frac{2x}{x^2+y^2}\rho_x
 -\frac{2y}{x^2+y^2}\rho_y
 -2(2-d)\frac{x^2-y^2}{(x^2+y^2)^2}\rho=0.
\end{equation}
\end{lemma}

\begin{proof}
Let $0\le f\in C_c^\infty(\HH)$ be nonzero and put $K=\operatorname{supp}f$.  Choose $t>0$ so small that $2\sqrt t<\dist(K,\RR)$.
The deterministic height bound for a hull of half-plane capacity $2t$ implies that no point of $K$ has been hit or swallowed by time $t$, for any driving function.  Taking conditional expectation with respect to $\mathcal F_t$ and then unconditional expectation in \eqref{eq:future-pushforward} and using the fact that $\widetilde\mu_H$ has the same law as $\mu_H$ and is independent of $\mathcal F_t$ therefore gives
\begin{equation}\label{eq:reverse-invariance-forward}
 \int f(z)\,\nu(dz)
 =\EE\int f(F_t(w))|F_t'(w)|^d\,\nu(dw).
\end{equation}

 Fix a terminal time $t_0\in(0,\infty)$ and set
\[
 h_t(z)=g_{t_0-t}(g_{t_0}^{-1}(z)),\qquad 0\le t\le t_0.
\]
Then $h_0(z)=z$, $h_{t_0}(z)=g_{t_0}^{-1}(z)$, and $h_t$ satisfies the reversed Loewner equation:
\[
 \partial_t h_t(z)=-\frac{2}{h_t(z)-U_{t_0-t}}.
\]
Let $\widehat U_t=U_{t_0}-U_{t_0-t}$ and $Z_t=h_t(w+U_{t_0})-U_{t_0-t}$. Then $(\widehat U_t)_{0\le t\le t_0}$ has the law of $(\sqrt\kappa B_t)_{0\le t\le t_0}$, and $Z_t$ satisfies
\[
 dZ_t=-\frac{2}{Z_t}\,dt+d\widehat U_t,\qquad Z_0=w.
\]
Thus
\BGE (F_{t_0}(w),F_{t_0}'(w)) \stackrel{d}{=} (Z_{t_0}^w,\partial_w Z_{t_0}^w).\label{law}\EDE
In fact, this equality holds with $t_0$ replaced by any $t\in [0,t_0]$.
For the following calculations, we define $\widehat B_t=\widehat U_t/\sqrt\kappa$, which is a standard Brownian motion on $[0,t_0]$.
Write $Z_t=X_t+iY_t$ and $J_t=|\partial_w Z_t|$.  Then, for $t\in [0,t_0]$,
\begin{align}
 dX_t&=-\frac{2X_t}{X_t^2+Y_t^2}\,dt+\sqrt\kappa\,d\widehat B_t,\nonumber\\
 dY_t&= \frac{2Y_t}{X_t^2+Y_t^2}\,dt,\label{eq:reverse-y}\\
 d\log J_t&
 = \frac{2(X_t^2-Y_t^2)}{(X_t^2+Y_t^2)^2}\,dt.\nonumber
\end{align}
Let $\mathcal A$ denote the differential operator:
\begin{equation}\label{eq:reverse-operator}
 \mathcal A
 =\frac\kappa2\partial_{xx}
 -\frac{2x}{x^2+y^2}\partial_x
 +\frac{2y}{x^2+y^2}\partial_y
 +2d\frac{x^2-y^2}{(x^2+y^2)^2}.
\end{equation}
Then  It\^o's formula and the compact support of $f$ give
\[
 \left.\frac{d}{dt}\right|_{t=0}
 \EE_w[J_t^d f(Z_t)]=\mathcal Af(w).
\]
Using \eqref{eq:reverse-invariance-forward}, which holds for sufficiently small $t$, \eqref{law}, which holds with $t_0$ replaced by $t\in [0,t_0]$, and differentiating at $t=0$, with the interchange of differentiation and integration justified below, yields
\begin{equation}\label{eq:distributional-A}
 \int_{\HH}\mathcal Af\,d\nu=0,
 \qquad f\in C_c^\infty(\HH)\text{ and }f\ge 0.
\end{equation}
The identity extends to all $f\in C_c^\infty(\HH)$ by
linearity, since every such function is a difference of
two nonnegative functions in $C_c^\infty(\HH)$.
So $\mathcal A^*\nu=0$ in the sense of distributions.

To justify the interchange, set $P_tf(w)=\EE_w[J_t^d f(Z_t)]$
and $Q_t=(P_tf-f)/t$. We show that $Q_t\to\mathcal Af$
in $L^1(\nu)$. Set $\eta=\min_K\Im z$, $M=\max_K\Im z$,
and choose $R_0$ with $K\subset\{|\Re z|\le R_0\}$.
By \eqref{eq:reverse-y}, $Y_t$ is nondecreasing and
$\frac{d}{dt}Y_t^2\le4$. Thus, for
$0<t\le t_0<\eta^2/8$, the condition $Z_t^w\in K$
requires $\eta/\sqrt2\le\Im w\le M$.
Consequently, both $Q_t$ and $\mathcal Af$ vanish outside
the fixed strip $S=\{w:\eta/\sqrt2\le\Im w\le M\}$.

For starting points $w\in S$, the whole reverse path satisfies
$Y_t\ge\eta/\sqrt2$, so the drift and the potential in
\eqref{eq:reverse-operator} are uniformly bounded.
Hence $|Z_t^w-w|\le\sqrt\kappa|\widehat B_t|+Ct$ and
$|J_t^d-1|\le Ct e^{Ct}$, with constants independent of $w$.
It\^o's formula gives
$Q_t(w)=t^{-1}\int_0^t\EE_w[J_s^d\mathcal Af(Z_s)]\,ds$.
Since $\mathcal Af$ is bounded and globally Lipschitz,
these estimates imply
$\sup_{w\in S}|Q_t(w)-\mathcal Af(w)|\le C(\sqrt t+t)$.
Thus $Q_t\to\mathcal Af$ uniformly on $S$.
For each fixed $R<\infty$, local finiteness of $\nu$ gives
convergence in $L^1(\nu)$ on the compact rectangle
$S\cap\{|\Re w|\le R\}$.
To obtain convergence on the whole strip, we control the
tails as follows.

For $w=x+iy$ in the strip with $|x|$ sufficiently large, $f(w)=\mathcal Af(w)=0$, and reaching $K$ requires a Brownian displacement of at least $|x|-R_0-Ct $. Using the integral representation of $Q_t$, the inclusion
$\operatorname{supp}(\mathcal Af)\subset K$, and the
Gaussian tail bound gives
\[
 |Q_t(x+iy)|\le C
 \exp\!\left(-\frac{(|x|-R_0-Ct_0)^2}{Ct}\right)
 \le C_1e^{-c_1x^2},\qquad 0<t\le t_0,
\]
where the last bound holds uniformly for sufficiently large $|x|$, and $C_1,c_1$ may depend on $t_0$.  The polynomial growth in \eqref{eq:nu-scaling} makes this bound $\nu$-integrable on the strip.  Dominated convergence on the tails, together with the compact-set convergence, proves the asserted $L^1(\nu)$ convergence and hence \eqref{eq:distributional-A}.

It remains to use hypoellipticity.  Ignoring the zero-order term in \eqref{eq:reverse-operator}, introduce the vector fields
\[
 X_1=\sqrt\kappa\,\partial_x,
 \qquad
 X_0=-\frac{2x}{x^2+y^2}\partial_x
      +\frac{2y}{x^2+y^2}\partial_y.
\]
For $x\ne0$, the $\partial_y$ component of $[X_1,X_0]$ is $-4\sqrt\kappa\,xy/(x^2+y^2)^2$, so $X_1$ and $[X_1,X_0]$ span the tangent plane.  At $x=0$ this component vanishes, but the $\partial_y$ component of $[X_1,[X_1,X_0]]$ equals $-4\kappa/y^3\ne0$.
Thus H\"ormander's bracket condition holds at every point of $\HH$.  The formal adjoint $\mathcal A^*$ has the same second-order field and the opposite first-order field, hence satisfies the same bracket condition; zero-order terms do not affect hypoellipticity.  H\"ormander's theorem \cite{Hormander1967} now implies that every distributional solution of $\mathcal A^*u=0$ is smooth.  Applied to the Radon measure $u=\nu$, this gives $\nu(dz)=\rho(z)A(dz)$ with $\rho\in C^\infty(\HH)$.  The distributional equation $\mathcal A^*\nu=0$ is therefore classical; expanding the formal adjoint gives \eqref{eq:rho-pde}.
\end{proof}

\begin{proposition} \label{prop:mean-green}
Let $G$ be as in \eqref{eq:green}.
There exists a  constant $C_\kappa\in(0,\infty)$ such that
\begin{equation}\label{eq:mean-green}
 \rho(z)=C_\kappa G(z) ,\qquad z\in\HH.
\end{equation}
\end{proposition}

\begin{proof}
Scaling \eqref{eq:nu-scaling} and reflection \eqref{eq:nu-reflection} together with the continuity of $\rho$ imply
\begin{equation}\label{eq:rho-scale-reflect}
 \rho(rz)=r^{d-2}\rho(z),\qquad
 \rho(-x+iy)=\rho(x+iy).
\end{equation}

By homogeneity, write $\rho(x+iy)=y^{d-2}q(u)$ with $u=x/y$.  Substitution into \eqref{eq:rho-pde} gives
\begin{equation}\label{eq:q-ode}
 \frac\kappa2 q''(u)
 +\frac{4u}{1+u^2}q'(u)
 +\frac{4(2-d)}{(1+u^2)^2}q(u)=0.
\end{equation}
One solution of \eqref{eq:q-ode} is $q_1(u)=(1+u^2)^{1/2-4/\kappa}$.  Reduction of order gives a second independent solution
\begin{equation*}
 q_2(u)=q_1(u)
 \int_0^u(1+s^2)^{4/\kappa-1}\,ds.
\end{equation*}
The first solution is even and the second is odd.  The reflection symmetry in \eqref{eq:rho-scale-reflect} says that $q$ is even, so its $q_2$ coefficient vanishes.  Hence for some constant $c\ge 0$,
\[
 \rho(z)=c  y^{d-2}(1+(x/y)^2)^{1/2-4/\kappa}
 =c (\Im z)^{d-2}(\sin\arg z)^{8/\kappa-1}.
\]
Using \eqref{eq:green}, we obtain \eqref{eq:mean-green} for some finite $C_\kappa\ge0$.

We now prove that $C_\kappa>0$.  By Theorem \ref{thm:exact-input}, $\mathcal H^{h_\kappa}_\Gamma(\gamma([1,2]))>0$ almost surely.  Since $\gamma([1,2])$ is bounded, there is $n\in\mathbb N$ such that $\PP[\mathcal H^{h_\kappa}_\Gamma(\gamma([1,2])\cap B(0,n))>0]>0$, which further implies that $\PP[\mathcal H^{h_\kappa}_\Gamma( B(0,n))>0]>0$, which implies that $\nu(B(0,n))>0$. Hence $C_\kappa>0$.
\end{proof}

%\begin{remark}\label{rem:constant}
%For every bounded Borel set $D\subset\HH$ with $\int_DG(z)A(dz)>0$,
%$$ %\begin{equation}\label{eq:C-ratio}
% C_\kappa=
% \frac{\EE\bigl[\mathcal H^{h_\kappa}(\Gamma\cap D)\bigr]}
% {\int_DG(z)\,A(dz)}.
%$$
%Thus the constant is characterized by a one-point mean normalization.  The present argument does not evaluate it %explicitly.
%\end{remark}
%\end{comment}

\section{Conditional expectation and finite-variation uniqueness}

% We next upgrade the mean identity to a conditional identity.

\begin{proposition} \label{prop:haus-compensator}
Let $D\Subset\HH$ be open.  Then
\begin{equation}\label{eq:haus-DM}
 \Theta_t^{\rm H}(D)+C_\kappa\Psi_t(D)\quad\text{is an }\mathcal F\text{-martingale.}
\end{equation}
\end{proposition}

\begin{proof}
Apply Lemma \ref{lem:future-measure} to $f=\mathbf{1}_D$ and condition on $\mathcal F_t$.  Since $\til\mu_H$ is independent of $\F_t$ and its mean measure is $C_\kappa G(w)A(dw)$ by Proposition \ref{prop:mean-green},
\BGE \label{condition}
 \EE[\Theta_\infty^{\rm H}(D)-\Theta_t^{\rm H}(D)\mid\F_t]
 =C_\kappa\int_{F_t^{-1}(D\cap H_t)}|F_t'(w)|^dG(w)\,A(dw).
\EDE
Change variables $z=F_t(w)$.  By \eqref{GD},
\( |F_t'(w)|^{d-2}G(w)
 =|g_t'(z)|^{2-d}G(g_t(z)-U_t)
 =G_{H_t}(z;\gamma(t),\infty).
\)
Thus, the righthand side of \eqref{condition} equals $C_\kappa\Psi_t(D)$ by \eqref{Psi}.  

Adding $\Theta_t^{\rm H}(D)$ to both sides of \eqref{condition} yields \eqref{eq:haus-DM}.
\end{proof}

% We can now prove the main theorem.

\begin{proof}[Proof of Theorem \ref{thm:main}]
Fix a bounded open rational rectangle $D\Subset\HH$.  By \eqref{eq:LR-doob-meyer} and \eqref{eq:haus-DM}, $N_t(D):=C_\kappa^{-1}\Theta_t^{\rm H}(D)-\Theta_t^{\rm O}(D)$ is a martingale.

Both $\Theta_t^{\rm H}(D)$ and $\Theta_t^{\rm O}(D)$ are continuous nondecreasing processes.  Continuity of the first is Lemma \ref{lem:clock-continuity}; continuity of the second follows from \cite{LRMinkowski}.  Thus $N(D)$ is a continuous martingale whose paths have locally finite variation, which must be a.s.\ constant.  Since $N_0(D)=0$,
\begin{equation}\label{eq:D-equality-time}
 \Theta_t^{\rm H}(D)=C_\kappa\Theta_t^{\rm O}(D)
 \qquad\text{for all }t\ge0,\quad\text{almost surely.}
\end{equation}

Take the countable determining class of bounded open rational rectangles whose closures lie in $\HH$, and intersect the corresponding probability-one events.  Letting $t\to\infty$ in \eqref{eq:D-equality-time} and using monotone convergence gives $\Haus(D)=C_\kappa\Occ(D)$ simultaneously on this determining class.  Hence
\begin{equation}\label{eq:measure-H}
 \left.\Haus\right|_{\HH}=C_\kappa\left.\Occ\right|_{\HH},\quad\text{almost surely.}
\end{equation}

By Lemma \ref{lem:real-zero}, $\Haus(\RR)=0$.  The occupation measure $\Occ$ also gives zero mass to $\RR$ by \cite{LRMinkowski}.  Both measures are supported on $\Gamma\subset\overline\HH$.  Thus \eqref{eq:measure-H} extends to all of $\CC$.  The curve-segment assertion follows from \eqref{eq:LR-segment} on the same probability-one event, completing the proof.
\end{proof}

\section{Consequences and a general identification principle}

\begin{corollary}[Double points]\label{cor:double}
Let $\mathrm{Doub}=\{z:\gamma(s)=\gamma(t)=z\text{ for some }0\le s<t<\infty\}$ be the set of SLE double points.  Then $\mathcal H^{h_\kappa}_\Gamma(\mathrm{Doub})=0$ almost surely.
\end{corollary}

\begin{proof}
By \eqref{eq:LR-segment},   $\Occ(\Gamma_q\cap\gamma([q,n]))=0$ for every rational $q>0$ and integer $n>q$.  These null sets cover $\mathrm{Doub}$.  Their countable union is $\Haus$-null by Theorem \ref{thm:main}, proving the claim.
\end{proof}

\begin{remark}[Other exact gauges]\label{rem:general}
The same proof of Theorem \ref{thm:main} applies to any
gauge $h$ satisfying \eqref{eq:regular-variation} with $q=d$
and for which \eqref{eq:positive-finite-segments} and
\eqref{eq:local-expectation} hold with $h_\kappa$ replaced by $h$. Indeed, \eqref{eq:regular-variation} gives Lemma \ref{lem:conformal-covariance} and standard regular-variation theory (equivalently, Potter bounds) also gives $h(r)/r^{d'}\to0$ for any $d'\in (1,d)$, while \eqref{eq:positive-finite-segments} and \eqref{eq:local-expectation} give local finiteness of the mean measure and finiteness and positivity of the identification constant.  Hence Theorem \ref{thm:main} holds for such $h$.
\end{remark}

\section*{Use of artificial intelligence}
The author used GPT~6 to explore proof ideas, locate relevant
references, and assist with the exposition.
The author simplified the proofs, reorganized the arguments, revised the exposition, carefully checked the mathematical details, and reviewed the final manuscript.
The author takes full responsibility for its contents.


\begin{thebibliography}{99}
\bibitem{Beffara}
V.~Beffara,
\emph{The dimension of the SLE curves},
Ann. Probab. \textbf{36} (2008), no.~4, 1421--1452.

\bibitem{Chistyakov2003}
D.~V. Chistyakov,
\emph{Fractal measures, p-adic numbers and continues transition between dimensions},
preprint (2003), \href{https://arxiv.org/abs/math/0302162}{arXiv:math/0302162}.

\bibitem{Hormander1967}
L.~H\"ormander,
\emph{Hypoelliptic second order differential equations},
Acta Math. \textbf{119} (1967), 147--171.
% DOI: 10.1007/BF02392081.

\bibitem{LRMinkowski}
G.~F. Lawler and M.~A. Rezaei,
\emph{Minkowski content and natural parameterization for the Schramm--Loewner evolution},
Ann. Probab. \textbf{43} (2015), no.~3, 1082--1120.
% DOI: 10.1214/13-AOP874.

\bibitem{LSnatural}
G.~F. Lawler and S.~Sheffield,
\emph{A natural parametrization for the Schramm--Loewner evolution},
Ann. Probab. \textbf{39} (2011), no.~5, 1896--1937.
% DOI: 10.1214/10-AOP560.

\bibitem{LZnatural}
G.~F. Lawler and W.~Zhou,
\emph{SLE curves and natural parametrization},
Ann. Probab. \textbf{41} (2013), no.~3A, 1556--1584.
% DOI: 10.1214/12-AOP742.

\bibitem{OAIexact}
OpenAI,
\emph{An explicit exact Hausdorff gauge for SLE:
A regular formula from quantitative tails and dense visits},
preprint, September 26, 2026,
\href{https://github.com/openai/math/blob/main/preprints/An-explicit-exact-Hausdorff-gauge-for-SLE-September-26-2026/An-explicit-exact-Hausdorff-gauge-for-SLE-September-26-2026.pdf}
{OpenAI Math repository}.

\bibitem{RS}
S.~Rohde and O.~Schramm,
\emph{Basic properties of SLE},
Ann. of Math. (2) \textbf{161} (2005), no.~2, 883--924.

\end{thebibliography}
\end{document}